\documentclass[11pt]{article}
\usepackage{amsmath,amssymb,amsbsy,amsfonts,amsthm,
               latexsym,amsopn,amstext,amsxtra,euscript,amscd}

\newtheorem{thm}{Theorem}

\newtheorem{prop}[thm]{Proposition}
\newtheorem{cor}[thm]{Corollary}
\newtheorem{corollary}[thm]{Corollary}
\newtheorem{lem}[thm]{Lemma}
\newtheorem{lemma}[thm]{Lemma}
\newtheorem{conj}[thm]{Conjecture}
\newtheorem{prob}[thm]{Problem}
\newtheorem{exa}[thm]{Example}
\newtheorem{rem}[thm]{Remark}

\newtheorem{defn}[thm]{Definition}
\newcommand{\ben}{\begin{enumerate}}
\newcommand{\een}{\end{enumerate}}
\newcommand{\ble}{\begin{lem}}
\newcommand{\ele}{\end{lem}}
\newcommand{\bth}{\begin{thm}}
\newcommand{\bpr}{\begin{prop}}
\newcommand{\epr}{\end{prop}}
\newcommand{\bpro}{\begin{prob}}
\newcommand{\epro}{\end{prob}}
\newcommand{\bco}{\begin{cor}}
\newcommand{\eco}{\end{cor}}
\newcommand{\bcon}{\begin{conj}}
\newcommand{\econ}{\end{conj}}
\newcommand{\bde}{\begin{defn}}
\newcommand{\ede}{\end{defn}}
\newcommand{\bex}{\begin{exa}}
\newcommand{\eex}{\end{exa}}

\newcommand{\barr}{\begin{array}}
\newcommand{\earr}{\end{array}}
\newcommand{\btab}{\begin{tabular}}
\newcommand{\etab}{\end{tabular}}
\newcommand{\beq}{\begin{equation}}
\newcommand{\eeq}{\end{equation}}
\newcommand{\bea}{\begin{eqnarray*}}
\newcommand{\eea}{\end{eqnarray*}}
\newcommand{\bce}{\begin{center}}
\newcommand{\ece}{\end{center}}
\newcommand{\bpi}{\begin{picture}}
\newcommand{\epi}{\end{picture}}
\newcommand{\bfi}{\begin{figure} \begin{center}}
\newcommand{\efi}{\end{center} \end{figure}}

\newcommand{\bsl}{\begin{slide}{}}

\newcommand{\bprob}{\begin{prob}}
\newcommand{\eprob}{\end{prob}}

\newtheorem{obs}[thm]{Observation}
\newcommand{\bobs}{\begin{obs}}
\newcommand{\eobs}{\end{obs}}

\def\cA{{\mathcal A}}

\def\cC{{\mathcal C}}
\def\cD{{\mathcal D}}

\newcommand{\BBN}{{\mathbb N}}

 \def\doublespacing{\parskip 5 pt plus 1 pt \baselineskip 18 pt
     \lineskip 13 pt \normallineskip 13 pt}

\usepackage{epsfig}
\include{psfig}

\begin{document}

\pagestyle{empty}

\title{\bf Minimal Niven numbers}

\author{\sc
 H. Fredricksen$^1$, E. J. Ionascu$^2$, F.~Luca$^3$, P. St\u anic\u a$^1$\\
\\
$^1$ \small Department of Applied Mathematics,
\small Naval Postgraduate School\\
\small Monterey, CA 93943, USA; \
\small{\em \{HalF,pstanica\}@nps.edu}\\
$^2$\small  Department of Mathematics,
\small Columbus State University\\
%\small 4225 University Avenue\\
\small Columbus, GA 31907, USA;
\small{\em ionascu\_eugen@colstate.edu}\\
$^3$\small  Instituto de Matem{\'a}ticas,
\small Universidad Nacional Aut\'onoma de M{\'e}xico \\
 \small C.P.~58089, Morelia, Michoac{\'a}n, M{\'e}xico;
\small  {\em fluca@matmor.unam.mx} }

\date{December $21^{st}$, 2007}

\maketitle

\renewcommand{\thefootnote}{}
\footnote{ \noindent\textbf{} \vskip0pt \textbf{Mathematics Subject
Classification:} 11L20, 11N25, 11N37 \vskip0pt\textbf{Key Words:} sum of digits,
Niven Numbers. \vskip0pt Work by F.~L. was started in the Spring of
2007 while he visited the Naval Postgraduate School. He would like
to thank this institution for its hospitality. H.~F.  acknowledges
support from the National Security Agency under contract RMA54.
Research of P.~S.\ was supported in part by a RIP grant from Naval
Postgraduate School.}

\begin{abstract}
Define $a_k$ to be the smallest positive multiple of $k$ such that
the sum of its digits in base $q$ is equal to $k$.
 The asymptotic behavior, lower and upper bound estimates
of $a_k$ are investigated. A characterization of the
minimality condition is also considered.
\end{abstract}

\doublespacing

\pagestyle{plain}

\section{Motivation}

A positive integer $n$ is a {\em Niven number} (or a {\em Harshad
number}) if it is divisible by the sum of its (decimal) digits. For
instance, 2007 is a Niven number since 9 divides 2007. A {\em
$q$-Niven number} is an integer $k$ which is divisible by the sum of its base
$q$ digits, call it $s_q(k)$ (if $q=2$, we shall use $s(k)$ for $s_2(k)$).
Niven numbers have been extensively studied by various
authors (see Cai \cite{Ca96}, Cooper and Kennedy \cite{CK93}, De
Koninck and Doyon \cite{KD}, De Koninck, Doyon and Katai
\cite{KDK03}, Grundman \cite{Gr94}, Mauduit, Pomerance and
S\'ark\"ozy \cite{MPS}, Mauduit and S\'ark\"ozy \cite{MS}, Vardi
\cite{Va91}, just to cite a few of the most recent works).

In this paper, we define a natural sequence in relation to $q$-Niven
numbers. For a fixed but arbitrary $k\in \BBN$ and a base $q\ge 2$,
one may ask whether or not there exists a $q$-Niven number whose sum
of its digits is precisely $k$. We will show later that the answer
to this is affirmative.
 Therefore, it makes sense to define $a_k$ to be the smallest positive multiple of $k$ such
that $s_q(a_k)=k$. In other words, $a_k$ is the smallest Niven
number whose sum of the digits is a given positive integer $k$. We
denote by $c_k$ the companion sequence
 $c_k={a_k}/{k}$, $k\in \BBN$. Obviously, $a_k$, respectively, $c_k$, depend on $q$, but we
will not make this explicit to avoid cluttering the notation.

In this paper we give constructive methods in Sections~\ref{sec:3},
\ref{sec:4} and \ref{sec:7} by two different techniques for the
binary and nonbinary cases,  yielding sharp upper bounds for $a_k$.
We find elementary upper bounds true for all $k$, and then better
nonelementary ones true for most odd $k$.

Throughout this paper, we use the Vinogradov symbols $\gg$ and $\ll$ and the Landau symbols $O$ and
$o$ with their usual meanings. The constants implied by such symbols are absolute.
We write $x$ for a large positive real number, and $p$ and $q$ for prime numbers.
If ${\cA}$ is a set of positive integers, we write ${\cA}(x)={\cA}\cap [1,x]$.
We write $\ln x$ for the natural logarithm of $x$ and $\log x=\max\{\ln x,1\}$.

\section{Easy proof for the existence of $a_k$}
\label{sec:2}

In this section we present a simple argument that shows that the above defined sequence $a_k$ is well
defined. First we assume that $k$ satisfies $\gcd(k,q)=1$. By
Euler's theorem, we can find an integer $t$ such that $q^{t}\equiv 1 \pmod k$,
and then define
\(
K=1+q^{t}+q^{2t}+\cdots+q^{(k-1)t}.
\)
Obviously, $K\equiv 0 \pmod k$,  and also
$s_q(K)=k$. Hence, in this case, $K$ is a Niven number whose digits
in base $q$ are only $0$'s and $1$'s and whose sum is $k$.

If $k$ is not coprime to $q$, we write $k=ab$ where
$\gcd(b,q)=1$ and $a$ divides $q^n$ for some $n\in \BBN$. As before,
we can find $K\equiv 0\pmod b$ with $s_q(K)=b$. Let
$u=\max\{n, \lceil{\log_q K}\rceil\}+1$, and define
\(
K'=(q^{u}+q^{2u}+\cdots +q^{ua}) K.
\)
Certainly $k=ab$ is a divisor of $K'$ and $s_q(K')=ab=k$. Therefore,
$a_k$ is well defined for every $k\in \BBN$.

This argument gives a large upper bound, namely of size
$\exp(O(k^2))$ for $a_k$.

We remark that if $m$ is the minimal $q$-Niven number corresponding
to $k$, then $q-1$ must divide $m-s_q(m)=kc_k-k=(c_k-1)k$. This
observation turns out to be useful in the calculation of $c_k$ for
small values of $k$. For instance, in base ten, the following table
of values of $a_k$ and $c_k$ can be established easily by using the previous simple observation. As an
example, if $k=17$ then $9$ has to divide $c_{17}-1$ and so we need only check $10,19,28$.

{\small
\begin{tabular}{|c|c|c|c|c|c|c|c|c|c|c|c|c|c|c|}
  \hline
$k$ & $10$& $11$ & $12$ & $13$& $14$& $15$& $16$& $17$& $18$& $19$& $20$& $21$ &$22$ & $23$\\

\hline $c_k$ & $19$ & $19$ &  $4$& $19$ & $19$  & $13$ & $28$& $28$
&$11$& $46$ &  $199$ &  $19$&  $109$&  $73$\\

\hline $a_k$ & $190$ & $209$ & $48$& $247$ & $266$  &
 $195$ & $448$& $476$ & $198$ & $874$ & $3980$ & $399$& $2398$& $1679$\\
\hline
\end{tabular}
}

\section{Elementary bounds for $a_k$ in the binary case}
\label{sec:3}

For each positive integer $k$ we set  $n_k= \lceil \log_2 k\rceil$.
Thus, $n_k$ is the smallest positive integer with $k\le  2^{n_k}$.
Assuming that $k\in \mathbb N$ ($k>1$) is odd, we let $t_k$ be the
multiplicative order of 2 modulo $k$, and so, $2^{t_k}\equiv 1\pmod
k$. Obviously, $t_k\ge n_k$ and $t_k\mid \phi(k)$, where $\phi$ is
Euler's totient function. Thus,
\begin{equation}\label{zero}
n_k \le t_k\le k-1.
\end{equation}

\begin{lem}
\label{thetwo} For every odd integer \ $k>1$, every integer\
$x\in\{0,1,\ldots, k-1\}$ can be represented as a sum modulo $k$ of
exactly\ $n_k$ distinct elements of $$\displaystyle
D=\{2^i\,|\,i=0,\ldots,n_k+k-2\}.$$
\end{lem}
\begin{proof} We find the required representation in a constructive way.
Let us start with an example. If $x=0$ and $k=2^{n_k}-1$, then since
$x\equiv k \pmod k$, we notice that in this case we have a
representation as required by writing $k=1+2+\cdots +2^{n_k-1}$
(note that $n_k-1\le n_k+k-2$ is equivalent to $k\ge 1$).

Any $x\in\{0,1,\ldots,k-1\}$ has at
most $n_k$ bits of which at most $n_k-1$ are ones. Next, let us
illustrate the construction when this binary representation of $x$
contains exactly $n_k-1$ ones, say
 \[
  x=2^{n_k-1}+2^{n_k-2} + \cdots + 2+1-2^j,~\text{for some}~j\in \{0,1,\ldots,n_k-1\}.
  \]
   First, we assume $j\le n_k-2$. Using $2^{j+1}=2^{j}+2^{j}\equiv 2^j+2^{j+t_k}\pmod k$, we write
\[
 x\equiv 2^{j+t_k}+2^{n_k-1}+\cdots+ 2^{j+2}+ 2^j+ 2^{j-1}+\cdots+ 1\pmod
 k,
\]
where both
 $j+t_k\le n_k-2+k-1=n_k+k-3$  and $j+t_k>n_k-1$
 are true according to (\ref{zero}). Therefore all exponents are
 distinct and they are contained in the required range, which gives us a representation of $x$
 as a sum of exactly $n_k$
different elements of $D$ modulo~$k$.

If $j=n_k-1$, then $x=2^{n_k-1}-1$. We consider $x+k$
instead of $x$. By the definition of $n_k$, we must have $k\ge
2^{n_k-1}+1$. Hence, $x+k\ge 2^{n_k}$, which implies that the binary
representation of $x+k$ starts with $2^{n_k}$ and it has at most
$n_k$ ones. Indeed, if $s(x+k)\ge n_k+1$, then $x+k\ge
2^{n_k}+2^{n_k-1}+\cdots +2+1=2^{n_k+1}-1$, which in turn
contradicts the inequality $x+k\le k-1+k=2k-1\le 2^{n_k+1}-3$ since
$k$ is odd.
If   $s(x+k)=n_k$, then we are done ($k\ge 3$).
If $s(x+k)= n_k-1$, then we proceed as
before and observe that this time $j+t_k\le n_k+k-2$ for every $j\in
\{0,1,2,\ldots, n_k-1\}$ and $j+t_k> n_k$ if $j>0$ which is an
assumption that we can make because in order to obtain $n_k-1$ ones
two of the powers of $2$, out of $1,2,2^2,\ldots,2^{n_k-1}$, must be
missing.

If $s(x+k)< n_k-1$, then for every zero in the representation of $x+k$, which is
preceded by a one and followed by $\ell$ ($\ell\ge 0$) other zeroes,
we can fill out the zeros gap in the following way. If such a zero
is given by the coefficient of $2^j$, then we  replace $2^{j+1}$ by
$2^j+2^{j-1}+\cdots+2^{j-\ell}+2^{j-\ell+t_k}$. This will give
$\ell+2$ ones instead of a one and $\ell+1$ zeros. We fill out all
gaps this way with the exception of the gap
corresponding to the smallest power of 2 and $\ell\ge 1$, where in
order to insure the inequality $j'+t_k>n_k$ ($j'=j-\ell+1>0$) one
will replace $2^{j+1}$ by
$2^j+2^{j-1}+\cdots+2^{j-\ell+1}+2^{j-\ell+1+t_k}$. The result will
be a representation in which all the additional powers $2^{j'+t_k}$
will be distinct and the total number of powers of two is $n_k$. The
maximum exponent of these powers is at most $j'+t_k\le n_k+k-2$.

If the representation of $x$ starts with $2^{n_k-1}$, then the
technique described above can be applied directly to $x$ making sure
that all zero gaps are completely filled. Otherwise, we apply the previous
technique  to $x+k$.
\end{proof}

\begin{exa}
\label{exp1}
 Let  $k=11$. Then $n_{11}=4$ and $t_{11}=10$.
 Suppose that we want to
 represent $9$ as a sum of $4$ distinct terms modulo 11 from the set
 $D=\{1,2,\ldots,2^{13}\}$.  Since $9=2^3+1$, we have $9=2^2+2+2+1$, so $9\equiv 2^2+2+2^{11}+1 \pmod {11}$.
 If we want to represent $7=2^2+2^1+2^0$ then, since this representation does  not contain $2^3$,
 we look at $7+11=18=2^4+2=2^3+2^3+2=2^3+2^2+2^2+2$. Thus,
$7\equiv 2^3+2^2+2^{12}+2\pmod {11}$.
\end{exa}

We note that the representation given by Lemma~\ref{thetwo} is not
unique. If this construction is applied in such a way that the zero
left when appropriate is always the one corresponding to the largest
power of $2$,  we will obtain the largest of such representations.
In the previous example, we can fill out the smallest gap first and
leave a zero from the gap corresponding to $2^3$, so $7\equiv 18\equiv
2^4+2=2^3+2^3+1+1\equiv 2^3+2^{13}+1+2^{10} \pmod {11}$.

Recall that $2^{\alpha}\| m$ means that $2^{\alpha}\mid m$ but
$2^{\alpha+1}\nmid m$. We write $\mu_2(m)$ for the exponent
$\alpha$.

\begin{thm}
\label{thethree} For all positive integers $k$ and $\ell$, there
exists a positive integer $n$ having the following properties:
\begin{enumerate}
\item[$(a)$] $s(nk)=\ell k$,
\item[$(b)$] $n \le (2^{\ell k+n_k}-2^{\mu_2(k)})/k$.
\end{enumerate}
\end{thm}

\begin{proof}
 It is clear that if $k$ is a power of $2$, say $k=2^s$, then
we can take $n=2^{\ell k}-1$ and so
$s(kn)=s(2^s+2^{s+1}+\cdots+2^{s+\ell k-1})=\ell k$. In this
case, the upper bound in part $(b)$ is sharp since
${n_k}=s=\mu_2(k)$.

Furthermore, if $k$ is of the form $k=2^md$ for some positive
integers $m,d$ with odd $d\ge 3$, then assuming that we can find an
integer $n\le (2^{2^m\ell d+{n_k}'}-1)/d$, where ${n_k}'=\lceil
\log_2 d \rceil$, such that $s(nd)=2^m\ell d$,  then $nk$
satisfies condition $(a)$ since
$
s(nk)=s(2^mnd)=s(nd)=2^m\ell
d=\ell k.
$
We observe that condition $(b)$ is also satisfied in
this case, because $(2^{2^m\ell d+{n_k}'}-1)/d=(2^{\ell
k+{n_k}}-2^m)/k. $

Thus, without loss of generality, we may assume in what follows that
$k\ge 3$ is odd. Consider the integer $M=2^{\ell k+n_k}-1=1+2^1+
\cdots+2^{\ell k+n_k-1}$, and so, $s(M)=\ell k+n_k$.
  By Lemma~\ref{thetwo}, we can write
\begin{equation}
\label{one} M\equiv  2^{j_1}+ 2^{j_2}+\cdots+ 2^{j_{n_k}}\pmod {k},
\end{equation}
where $0\le j_1< j_2<\cdots <j_n\le k+n_k-2< \ell k+n_k-1$.
Therefore, we may take
\begin{equation*}\label{two}
n= \dfrac{M-(2^{j_1}+2^{j_2}+\cdots+2^{j_{n_k}})}{k},
\end{equation*}
which is an integer by  \eqref{one} and  satisfies
$
s(nk)=s(M-(2^{j_1}+2^{j_2}+\cdots+2^{j_{n_k}}))=%{n_k}+\ell
%k-{n_k}=
\ell k.
$
\end{proof}

\begin{cor}
\label{cor4} The sequence $(a_k)_{k\ge 1}$ satisfies
\begin{equation}
\label{eqfour} 2^k-1 \le a_k\le 2^{k+n_k}-2^{\mu_2(k)}.
\end{equation}
\end{cor}

\begin{proof}
 The first inequality in \eqref{eqfour} follows from the
fact that if $s(a_k)=k$,  then $a_k\ge 1+2+\cdots+2^{k-1}=2^k-1$.
The second inequality in \eqref{eqfour} follows from
Theorem~\ref{thethree} by taking $\ell=1$, and from the minimality
condition in the definition of $a_k$.
\end{proof}

We have computed $a_k$ and $c_k$ for all $k=1,\ldots, 128$,
\[
c_1=1,\ c_2=3,\ c_3=7,\ldots,\ c_{20}=209715,\ldots.
\]
and the graph of $k\to \ln(c_k)$ against the functions $k\to \ln(2^k)$
and $k\to \ln(2^k-1)-\ln(k)$ is included in Figure~\ref{fig:c-k}.

\begin{figure}
\begin{center}
\epsfig{file=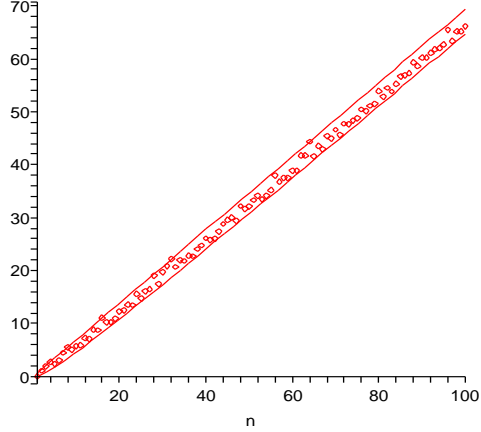,height=2.5in,width=2.8in}
  \caption{The graphs of $k\to \ln(c_k)$ and $k\to \ln(2^k)$, $k\to \ln(2^k-1)-\ln(k)$}
\label{fig:c-k}
\end{center}
\end{figure}

  The right hand side of inequality (\ref{eqfour})
is sharp when $k=2^s$, as we have already seen.  For $k=2^s-1$, we
get values of $c_k$ very close to $2^k-1$ but, in general, numerical
evidence shows that $c_k/2^k$ is closer  to zero more often than it
is to  $1$. In fact, we show in Section~\ref{sec:6} that this is
indeed the case at least for odd indices (see
Corollary~\ref{cor:X},  Corollary~\ref{cor:1} and relation~\eqref{averageconjecture}).

\section{Improving binary estimates and some closed formulae}
\label{closedformulae}
\label{sec:4}

 In order to obtain better bounds for $a_k$, we introduce the following classes of odd integers. For
 a positive integer $m$ we define
 {\small
\[
{\cal C}_m=\{k\equiv 1\pmod 2| \ 2^{k+m}-1\equiv \sum_{i=1}^m
2^{j_i} \pmod {k},\ \text{for}\ 0\le j_1<j_2<\cdots  <j_m\le
m+k-2\}.
\]
}
Let us observe that ${\cal C}_m\subset {\cal C}_{m+1}$. Indeed, if
$k$ is in ${\cal C}_m$, we then have $2^{k+m}-1\equiv 2^{j_1}+\cdots
+2^{j_m}$ for some $0\le j_1<j_2<\cdots  <j_m\le m+k-2$. Multiplying
by 2 the above congruence and adding one to both sides, we get
$2^{k+m+1}-1\equiv 1+2^{j_1+1}+\cdots +2^{j_m+1}$, representation
which  implies that $k$ belongs to ${\cal C}_{m+1}$. Note also
that Lemma~\ref{thetwo} shows that every odd integer $k\ge 3$
belongs to ${\cal C}_{u}$, where $u=\left\lceil \log k/\log
2\right\rceil$. Hence, we have
$
2\BBN+1=\bigcup_{m\in \BBN} {\cal C}_m.
$

\begin{thm}\label{c1} For every $k\in {\cal C}_1$, we have
$$
2^k-1< a_k < 2^{k+1}-1.
$$
In particular, $c_k/2^k\to 0$ as $k\to \infty$ through ${\cal C}_1$.
Furthermore, $a_k=2^{k+1}-1-2^{j_1}$, where $j_1=j_0+st_k$, with
$s=\lfloor (k-1-j_0)/t_k\rfloor$, and $0\le j_0\le t_k-1$ is such
that $2^{k+1}-1\equiv 2^{j_0}\pmod k$.
\end{thm}
\begin{proof}
We know that $2^k-1\not \equiv 0 \pmod k$
 (see \cite[Problem 37, p. 109]{NZM91}). Hence, an integer of binary length $k$
whose sum of digits is $k$ is not divisible by $k$. Therefore,
$a_k>2^k-1$.

Next, we assume that $a_k$ is an integer of binary length $k+1$ and
sum of digits $k$; that is, $a_k=2^{k+1}-1-2^j$ for some
$j=0,\ldots,k-1$. But $2^{k+1}-1\equiv x \pmod k$, and by hypothesis
there exists $j_0$ such that $x=2^{j_0}$ for some $j_0\in \{0,\ldots
,t_k-1\}$. In order to obtain $a_k$, we need to subtract the highest
power of 2 possible because of the minimality of $a_k$. So, we need
to take the greatest exponent $j_1=j_0+st_k\le k-1$, leading to
 $s=\lfloor (k-1-j_0)/t_k\rfloor$. Hence,
$a_k=2^{k+1}-1-2^{j_1}$.
\end{proof}

 Based on the above argument, we can compute, for instance,
$a_5=55=2^6-1-2^3$, since $2^3-1\equiv 2^3$ (mod 5). Similarly,
 $a_{29}=2^{30}-1-2^5=1073741791$, since  $2^{30}-1\equiv 2^5$ (mod
 29), and $a_{25}=2^{26}-1-2^{19}=66584575$, since
 $2^{26}-1\equiv 2^{19}$ (mod 25), or perhaps the more interesting
 example
 $a_{253}=2^{254}-1-2^{242}$.

\begin{thm}\label{cm}
If $m\in \BBN$, and $k\in {\cal C}_{m+1}\setminus {\cal C}_m$, we
then have
$$2^{k+m-1}-1< a_k < 2^{k+m}-1.$$
Thus, $c_k/2^k\to 0$ as $k\to \infty$ in ${\cal C}_{m}$ for any
fixed $m$.
\end{thm}
\begin{proof}
Similar as the proof of Theorem~\ref{c1}.
\end{proof}

\begin{thm}
\label{mersenne} For all integers $k=2^i-1\ge 3$, we have
\begin{equation}
\label{mers_bound} a_k\le 2^{k+ {k^-}}+2^k-2^{k-i}-1,
\end{equation}
where ${k^-}$ is the least positive residue of $-k$ modulo $i$.
Furthermore, the bound \eqref{mers_bound} is tight when  $k=2^i-1$
is a Mersenne prime. In this case,  we have $c_k/2^k\to 1/2$ as
$k\to \infty$ through Mersenne primes, assuming that this set is
infinite.
\end{thm}
\begin{proof}
For the first claim, we show that the sum of binary digits of the
bound of the upper bound on \eqref{mers_bound} is exactly $k$, and
also that this number is a multiple of $k$. From the definition of
$k^-$, we find that $k+k^-=i\alpha$ for some positive integer
$\alpha$. Since
\begin{eqnarray*}
2^{k+ {k^-}}+2^k-2^{k-i}-1
&=& 2^{k-i}(2^i-1)+2^{i\alpha}-1\\
&=&(2^i-1)(2^{k-i}+2^{i(\alpha-1)}+ 2^{i(\alpha-2)}+\cdots +1),
\end{eqnarray*}
we get that $2^{k+ {k^-}}+2^k-2^{k-i}-1$ is divisible by $k$.
Further, $k^-\ge 1$ since $k$ is not divisible by $i$ (see the proof
of Theorem~\ref{c1}), and
\begin{eqnarray*}
s\left(2^{k+ {k^-}}+2^k-2^{k-i}-1\right)
&=& s\left(2^{k+k^- -1}+\cdots +2+1+2^k-2^{k-i}\right)\\
&=& s\left(2^{k+k^- -1}+\cdots +2^k+\cdots+\widehat{2^{k-i}}+ \cdots+2+1+2^k\right)\\
&=& s\left(2^{k+k^-}+2^{k-1}+\cdots +\widehat{2^{k-i}}+ \cdots+2+1\right)\\
&=&k,
\end{eqnarray*}
where $\hat t$ means that $t$ is missing in that sum. The first
claim is proved.

We now consider a Mersenne prime $k=2^i-1$. First, we show that
$k\in {\cal C}_{i}\setminus {\cal C}_{i-1}$. Since $u=\lceil \log
k/\log 2\rceil=i$, by Lemma~\ref{thetwo}, we know that $k\in {\cal
C}_i$. Suppose by way of contradiction that $k\in {\cal C}_{i-1}$.
Then
\begin{equation}\label{important}
2^{k+i-1}-1\equiv 2^{j_1}+\cdots+2^{j_{i-1}} \pmod k
\end{equation}
holds with some $0\le j_1<j_2<\cdots  <j_{i-1}\le k+i-3$. Since $k$
is prime, we have that $2^{k-1}\equiv 1 \pmod k$, and so
$2^{k+i-1}-1\equiv 2^i-1\equiv 0\pmod k.$

Because $2^i\equiv 1 \pmod k$, we can reduce all powers $2^{j}$ of
$2$ modulo $k$ to powers with exponents less than or equal to $i-1$.
We get at most $i-1$ such terms. But in this case, the sum of at
least one and at most $i-1$  distinct members of the set
$\{1,2,\ldots,2^{i-1}\}$ is positive and less than the sum of all of
them, which is $k$. So, the equality (\ref{important}) is
impossible.

To finish the proof, we need to choose the largest representation
$x=2^{j_1}+\cdots +2^{j_{i}}$, with  $0\le j_1<j_2<\cdots <j_{i}\le
k+i-2$, such that  $2^{k+i}-1\equiv x \pmod k$. But $2^{k+i}-1\equiv
2^{i+1}-1\equiv 1 \pmod k$. Since the exponents $j$ are all
distinct, the way to accomplish this is to take $j_i=k+i-2$,
$j_{i-1}=k+i-3,\ldots, j_2=k$, and finally $j_1$ to be the greatest
integer with the property that the resulting $x$ satisfies $x\equiv
1 \pmod k$. Since $x=2^{j_1}+2^{k}(1+2+\cdots
+2^{i-2})=2^{j_1}+2^k(2^{i-1}-1)\equiv 2^{j_1}+2^i-2\equiv 2^{j_1}-1
\pmod k$, we need to have $2^{j_1}\equiv 2 \pmod k$.  Since the
multiplicative order of $2$ modulo $k$ is clearly $i$, we have to
take the largest $j_1=1+si$ such that $1+s i<k$. But $i$ must be
prime too and so $2^{i-1}\equiv 1\pmod i$. This implies
$k=2^i-1\equiv 1 \pmod i$. Therefore $j_1=k-i$. So,
$a_k=2^{k+i}-1-x=2^{k+i}-1-2^{k-i}-2^{k+i-1}+2^k=2^{k+i-1}+2^k-2^{k-i}-1$
and the inequality  given in our statement becomes an equality since
$k^{-}=i-1$ in this case.

 Regarding the limit claim, we observe that
$$\frac{c_k}{2^k}=\frac{k+1}{2k}+\frac{1}{k}-\frac{1}{k2^i}-\frac{1}{k2^k}\
\longrightarrow\ \frac{1}{2},$$ as $i$ (and as a result $k$)  goes
to infinity.
\end{proof}

Between the two extremes, Theorems~\ref{cm} and \ref{mersenne}, we
find out that the first situation is more predominant (see
Corollary~\ref{cor:1}). Next, we give quantitative results on the
sets ${\cal C}_m$. However we start with a result which shows that
${\cal C}_1$ is of asymptotic density zero as one would less expect.

\section{${\cal C}_1$ is of density zero}
\label{sec:5}

Here, we show that ${\cal C}_1$ is of asymptotic density zero. For
the purpose of this section only, we omit the index and simply write
$$\cC=\{1\le n: 2^{n+1}-1\equiv 2^j\pmod n~{\text{\rm
for~some}}~j=1,2,\ldots\}.
$$
It is clear that $\cC$ contains only odd numbers. Recall that for a positive
real number $x$ and a set ${\cal A}$ we put ${\cal A}(x)={\cal
A}\cap [1,x]$. We prove the following estimate.

\begin{thm}
\label{thm:1} The estimate
$$
\#\cC(x)\ll \frac{x}{(\log\log x)^{1/7}}
$$
holds for all $x>e^{e}$.
\end{thm}

\begin{proof}
We let $x$ be large, and put $q$ for the smallest prime exceeding
$$y=\frac{1}{2}\left(\frac{\log\log x}{\log\log\log x}\right)^{1/2}.
$$
Clearly, for large $x$ the prime $q$ is odd and its size is
$q=(1+o(1))y$ as $x\to\infty$. For an odd prime $p$ we write $t_p$
for the order of $2$ modulo $p$ first defined at the beginning of
Section~\ref{sec:3}. Recall that this is the smallest positive
integer $k$ such that $2^k\equiv 1\pmod p$. Clearly, $t_p\mid p-1$.
We put
\begin{equation}
\label{set:P}
{\cal P}=\{p~{\text{\rm prime}}:p\equiv 1\pmod q~{\text{\rm
and}}~t_p\mid (p-1)/q\}.
\end{equation}
The effective version of Lagarias and Odlyzko of Chebotarev's
Density Theorem (see \cite{LO}, or page 376 in \cite{Pap}), shows
that there exist absolute constants $A$ and $B$ such that the
estimate
\begin{equation}
\label{eq:Pt} \# {\cal
P}(t)=\frac{\pi(t)}{q(q-1)}+O\left(\frac{t}{\exp\left(A{\sqrt{\log
t}}/q\right)}\right)
\end{equation}
holds for all real numbers $t$ as long as $q\le B(\log t)^{1/8}$. In
particular, we see that estimate \eqref{eq:Pt} holds when $x>x_0$ is
sufficiently large and uniformly in $t\in [z,x]$, where we take
$z=\exp((\log\log x)^{100})$.

We use the above estimate to compute the sum of the reciprocals of
the primes $p\in {\cal P}(u)$, where we put $u=x^{1/100}$. We have
$$
S=\sum_{p\in {\cal P}(u)}\frac{1}{p}=\sum_{\substack{p\in {\cal P}\\
p\le z}}\frac{1}{p}+\sum_{\substack{p\in {\cal P}\\ z< p\le
u}}\frac{1}{p}=S_1+S_2.
$$
For $S_1$, we only use the fact that every prime $p\in {\cal P}$ is
congruent to $1$ modulo $q$. By the Brun-Titchmarsh inequality we
have
$$
S_1\le \sum_{\substack{p\le z\\ p\equiv 1\pmod q}}\frac{1}{p}\ll
\frac{\log\log z}{\phi(q)}\ll \frac{\log\log\log x}{q}=O(1).
$$
For $S_2$, we are in the range where estimate \eqref{eq:Pt} applies
so by Abel's summation formula
\begin{eqnarray*}
S_2 & = & \sum_{\substack{p\in {\cal P}\\ z\le p\le
u}}\frac{1}{p}\ll \int_z^u\frac{d\#{\cal P}(t)}{t}
 = \frac{\#{\cal P}(t)}{t}\Big|_{t=z}^{t=u}\\
 & + & \int_{z}^u
\left(\frac{\pi(t)}{q(q-1)t^2}+O\left(\frac{t}{\exp(A{\sqrt{\log
t}}/q)}\right)\right) dt\\
& = & \int_z^u \frac{dt}{q(q-1)t\log t}+
O\left(\frac{1}{q^2}\right)+O\left(\int_z^u\frac{dt}{q(q-1)t(\log
t)^2}\right)\\
& = & \frac{\log\log u-\log\log
z}{q(q-1)}+O\left(\frac{1}{q^2}\right)
 =  \frac{\log\log x}{q(q-1)}+ O(1).
\end{eqnarray*}
In the above estimates, we used the fact that
$$
\pi(t)=\frac{t}{\log t}+O\left(\frac{t}{(\log t)^2}\right),
$$
as well as the fact that
$$
\frac{t}{\exp(A{\sqrt{\log t}}/q)}=O\left(\frac{t}{q^2(\log
t)^2}\right)
$$
uniformly for $t\ge z$. To summarize, we have that
\begin{equation}
\label{eq:sumoverP} S=\frac{\log\log x}{q(q-1)}+O(1)=\frac{\log\log
x}{q^2} +O\left(\frac{\log\log x}{q^3}+1\right)=\frac{\log\log
x}{q^2}+O(1).
\end{equation}

We next eliminate a few primes from ${\cal P}$ defined in \eqref{set:P}. Namely, we let
$$
{\cal P}_1=\{p:t_p<p^{1/2}/(\log p)^{10}\},
$$
and
$$
{\cal P}_2=\{p:p-1~{\text{\rm has~a~divisor}}~d~{\text{\rm
in}}~[p^{1/2}/(\log p)^{10}, p^{1/2}(\log p)^{10}]\}.
$$
A well-known elementary argument (see, for example, Lemma 4 in
\cite{BaGaLuSh}) shows that
\begin{equation}
\label{eq:P1} \#{\cal P}_1(t)\ll \frac{t}{(\log t)^2},
\end{equation}
therefore by the Abel summation formula one gets easily that
$$
\sum_{p\in {\cal P}_1}\frac{1}{p}=O(1).
$$
As for ${\cal P}_2$, results of Indlekofer and Timofeev from
\cite{IT} show that
$$
\#{\cal P}_2(t)\ll \frac{t\log\log t}{(\log t)^{1+\delta}},
$$
where $\delta=2-(1+\log\log 2)/\log 2=0.08\ldots$, so again by
Abel's summation formula one gets that
$$
\sum_{p\in {\cal P}_2}\frac{1}{p}=O(1).
$$
We thus arrive at the conclusion that letting ${\cal Q}={\cal
P}\backslash ({\cal P}_1\cup {\cal P}_2)$, we have
\begin{equation}
\label{eq:S'} S'=\sum_{p\in {\cal Q}(u)}\frac{1}{p}=S-\sum_{p\in
{\cal P}_1(u)\cup {\cal P}_2(u)}\frac{1}{p}=\frac{\log\log
x}{q^2}+O(1).
\end{equation}
Now let us go back to the numbers $n\in {\cal C}$. Let ${\cal D}_1$
be the subset of ${\cal C}(x)$ consisting of the numbers free of
primes in ${\cal Q}(u)$. By the Brun sieve,
\begin{eqnarray}
\label{eq:C1} \#{\cal D}_1 & \ll & x\prod_{p\in {\cal
Q}(u)}\left(1-\frac{1}{p}\right) = x\exp\left(-\sum_{p\in {\cal
Q}(u)}\frac{1}{p}+O\left(\sum_{p\in {\cal
Q}(u)}\frac{1}{p^2}\right)\right)\nonumber\\ &\ll&
x\exp(-S'+O(1))\ll
x\exp\left(-\frac{\log\log x}{q^2}\right)\nonumber\\
& = & \frac{x}{(\log\log x)^{4+o(1)}}\ll \frac{x}{(\log\log x)^3}.
\end{eqnarray}
Assume from now on that $n\in {\cal C}(x)\backslash {\cal D}_1$.
Thus, $p\mid n$ for some prime $p\in {\cal Q}(u)$. Assume that
$p^2\mid n$ for some $p\in {\cal Q}(u)$. Denote by $\cD_2$ the
subset of such $n\in \cC(x)\backslash \cD_1$. Keeping $p\in {\cal
Q}(u)$ fixed, the number of $n\le x$ with the property that $p^2\mid
n$ is $\le x/p^2$. Summing up now over all primes $p\equiv 1\pmod q$
not exceeding $x^{1/2}$, we get that the number of such $n\le x$ is
at most
\begin{equation}
\label{eq:cC2} \#\cD_2\le \sum_{\substack{p\le x^{1/2}\\ p\equiv
1\pmod q}} \frac{x}{p^2}\ll \frac{x}{q^2\log q}\ll \frac{x}{\log\log
x}.
\end{equation}
Let $\cD_3=\cC(x)\backslash (\cD_1\cup \cD_2)$. Write $n=pm$, where
$p$ does not divide $m$. We may also assume that $n\ge x/\log x$
since there are only at most $x/\log x$ positive integers failing
this condition. Put $t=t_p$. The definition of ${\cal C}$ implies
that
$$
2^{mp+1}\equiv 2^j+1\pmod p
$$
for some $j=1,2,\ldots,t$, and since $2^p\equiv 2\pmod p$, we get
that $2^{mp+1}\equiv 2^{m+1}\pmod p$. We note that $2^{m+1}\pmod p$
determines $m\le x/p$ uniquely modulo $t$. We estimate the number of
values that $m$ can take modulo $t$. Writing $X=\{2^j\pmod p\}$, we
see that $\#\{m\pmod p\}\le I/t$, where $I$ is the number of
solutions $(x_1,x_2,x_2)$ to the equation
\begin{equation}
\label{eq:*}
 x_1-x_2-x_3=0,\qquad x_1,~x_2,~x_3\in X.
\end{equation}
Indeed, to see that, note that if $m$ and $j$ are such that
$2^{m+1}\equiv 1+2^j\pmod p$, then
$(x_1,x_2,x_3)=(2^{m+1+y},2^y,2^{j+y})$ for $y=0,\ldots,t-1$, is
also a solution of equation \eqref{eq:*}, and conversely, every
solution $(x_1,x_2,x_3)=(2^{y_1},2^{y_2},2^{y_3})$ of equation
\eqref{eq:*} arises from $2^{m+1}\equiv 1+2^{j}\pmod p$, where
$m+1=y_1-y_2$ and $j=y_3-y_2$, by multiplying it with $2^{y_2}$.

To estimate $I$, we use exponential sums. For a complex number $z$
put ${\bf e}(z)=\exp(2\pi i z)$. Using the fact that for $z\in
\{0,1,\ldots,p-1\}$ the sum
$$
\frac{1}{p}\sum_{a=0}^{p-1}{\bf e}(a z/p)
$$
is $1$ if and only if $z=0$ and is $0$ otherwise, we get
$$
I=\frac{1}{p}\sum_{x_1,x_2,x_3\in X}\sum_{a=0}^{p-1}{\bf
e}(a(x_1-x_2-x_3)/p).
$$
Separating the term for $a=0$, we get
$$
I  =  \frac{(\# X)^3}{p}+\frac{1}{p}\sum_{a=1}^{p-1}
\sum_{x_1,x_2,x_3\in X} {\bf e}(a(x_1-x_2-x_3)/p)=\frac{t^3}{p}+
\frac{1}{p}\sum_{a=1}^{p-1} T_a T_{-a}^2,
$$
where we put
$
T_{a}=\sum_{x_1\in X} {\bf e}(ax_1/p).
$
A result of Heath-Brown and Konyagin \cite{HBK}, says that if $a\ne
0$, then
$$
|T_a|\ll t^{3/8}p^{1/4}.
$$
Thus,
$$
I=\frac{t^3}{p}+O(t^{9/8}p^{3/4}),
$$
leading to the fact that the number of values of $m$ modulo $t$ is
$$
\#\{m\pmod t\}\le \frac{I}{t}\le \frac{t^2}{p}+O(t^{1/8}p^{3/4}).
$$
Since also $m\le x/p$, it follows that the number of acceptable
values for $m$ is
$$
\ll \frac{x}{pt}\left(\frac{t^2}{p}+t^{1/8}p^{3/4}\right)\ll
\frac{xt}{p^2}+\frac{x}{t^{7/8}p^{1/4}}
$$
(note that $x/pt\ge 1$ because $pt<p^2<u^2<x$). Hence,
$$
\#\cD_3\le \sum_{p\in {\cal Q}(u)}\frac{xt}{p^2}+\sum_{p\in {\cal
Q}(u)}\frac{x}{t^{7/8}p^{1/4}}=T_1+T_2.
$$
For the first sum $T_1$ above, we observe that $t\le p/q$, therefore
$t/p^2\le 1/(pq)$. Thus, the first sum above is
\begin{equation}
\label{eq:sum1} T_1 \ll \sum_{\substack{p\in {\cal
Q}(u)}}\frac{x}{pq}\ll \frac{xS'}{q}\ll \frac{x\log\log x}{q^3}\ll
x \frac{(\log\log\log x)^{3/2}}{(\log\log x)^{1/2}},
\end{equation}
where we used again estimate \eqref{eq:S'}. Finally, for the second
sum $T_2$, we change the order of summation and thus get that
\begin{equation}
\label{eq:**}
T_2\le x\sum_{t\ge t_0} \frac{1}{t^{7/8}}\sum_{\substack{p\in {\cal Q}(u)\\
t(p)=t}} \frac{1}{p^{1/4}},
\end{equation}
where $t_0=t_0(q)$ can be taken to be any lower bound on the
smallest $t=t_p$ that can show up. We will talk about it later. For
the moment, note that for a fixed $t$, $p$ is a prime factor of
$2^t-1$. Thus, there are only $O(\log t)$ such primes. Furthermore,
for each such prime we have $p>qt$. Hence,
$$
T_2\ll \frac{x}{q^{1/4}}\sum_{t\ge t_0}\frac{\log t}{t^{9/8}}.
$$
Since $p\not\in {\cal P}_1\cup {\cal P}_2$, we get that
$t_p>p^{1/2}(\log p)^{10}$. Since $p\ge 2q+1$, we get that $t\gg
q^{1/2}(\log q)^{10}$. Thus, for large $x$ we may take
$t_0=q^{1/2}(\log q)^9$ and get an upper bound for $T_2$. Hence,
\begin{eqnarray}
\label{eq:sum2}
 T_2 & \ll & \frac{x}{q^{1/4}}\sum_{t>q^{1/2}(\log q)^9}\frac{\log t}{t^{9/8}}\ll
\frac{x}{q^{1/4}}\int_{q^{1/2}(\log q)^9}^{\infty} \frac{\log
s}{s^{9/8}} \,
ds\nonumber \\
& \ll & \frac{x}{q^{1/4}} \left(-\frac{\log
s}{s^{1/8}}\Big|_{q^{1/2}(\log q)^9}^{\infty}\right)
 \ll   \frac{x}{q^{1/4+1/16}(\log q)^{1/8}}\ll \frac{x}{q^{5/16}(\log
q)^{1/8}}\nonumber\\
&  \ll & \frac{x(\log\log\log x)^{1/32}}{(\log\log x)^{5/32}}.
\end{eqnarray}

Combining the bounds \eqref{eq:sum1} and \eqref{eq:sum2}, we get
that
$$
\#\cD_3\ll \frac{x}{(\log\log x)^{1/7}},
$$
which together with the bounds \eqref{eq:C1} and \eqref{eq:cC2}
completes the proof of the theorem.
\end{proof}

\begin{figure}
\begin{center}
\epsfig{file=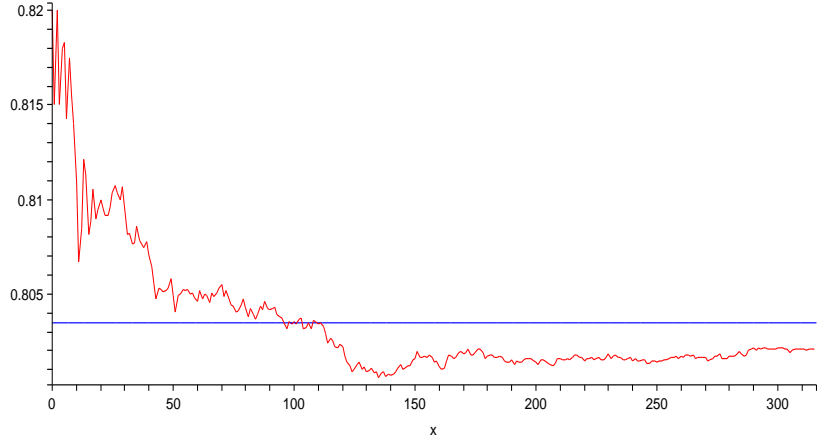,height=2.5in,width=5in}
  \caption{The graph of $2\frac{\#\cC_2(x)}{x}$, $1\le x\le 63201$, $x$ odd.}
\label{fig:densityofc2}
\end{center}
\end{figure}
Although the density of $\cC_1$ is zero, one my try to calculate the
densities of $\cC_m$ ($m>1$) hoping that they are positive and
approach $1$ as $m\to \infty$. In the Figure~\ref{fig:densityofc2}
we have numerically calculated the density of $\cC_2$ within the odd
integers up to 63201. Nevertheless, we abandoned this idea having
conjectured that the density of each $\cC_m$ is still zero. However,
the next section gives a way out to proving that $c_k/2^k$ goes to
zero in arithmetic average over odd integers $k$.

\section{The sets $\cC_m$ for large $m$}
\label{sec:6}

In this section, we prove the following result.

\begin{thm}
\label{thm:2} Put $m(k)=\lfloor \exp(4000(\log\log\log
k)^3)\rfloor$. The set of odd positive integers $k$ such that $k\in
{\cal C}_{m(k)}$ is of asymptotic density $1/2$.
\end{thm}

In particular, most odd positive integers $k$ belong to ${\cal
C}_{m(k)}$.

\begin{proof}
Let $x$ be large. We put
$$
y=(\log\log x)^3.
$$
We start by discarding some of the odd positive integers $k\le x$.
We start with
$$
{\cal A}_1=\{k\le x:q^2\mid k,~{\text{\rm or}}~q(q-1)\mid
k,~{\text{\rm or}}~q^2\mid \phi(k)~{\text{\rm for~some~prime}}~q\ge
y\}.
$$
Clearly, if $n\in {\cal A}_1$, then there exists some prime $q\ge y$
such that either $q^2\mid n$, or $q(q-1)\mid n$, or $q^2\mid p-1$
for some prime factor $p$ of $n$, or $n$ is a multiple of two primes
$p_1<p_2$ such that $q\mid p_i-1$ for both $i=1$ and $2$. The number
of integers in the first category is
\begin{eqnarray*}
&& \le \sum_{y<q\le x^{1/2}}\left\lfloor
\frac{x}{q^2}\right\rfloor\le x\sum_{y<q\le
x^{1/2}}\frac{1}{q^2}\ll x\int_y^{x^{1/2}} \frac{dt}{t^2}\ll
\frac{x}{y}=\frac{x}{(\log\log x)^3}=o(x)
\end{eqnarray*}
as $x\to \infty$. Similarly, the number of integers in the
second category is
$$
\le \sum_{y<q<x^{1/2}+1}\left\lfloor
\frac{x}{q(q-1)}\right\rfloor\ll x\sum_{y\le q\le
x^{1/2}+1}\frac{1}{q^2}\ll \frac{x}{y}= \frac{x}{(\log\log
x)^3}=o(x)
$$
as $x\to\infty$. The number of integers in the third category is
\begin{eqnarray*}
&& \le \sum_{y<q\le x^{1/2}}\sum_{\substack{p\le x\\ p\equiv 1 \pmod
{q^2}}} \left\lfloor\frac{x}{p}\right\rfloor\le x\sum_{y<q\le x^{1/2}}
\sum_{\substack{p\le x\\
p\equiv 1\pmod {q^2}}}\frac{1}{p}\\
&& \ll x\sum_{y<q\le x^{1/2}}\frac{\log\log x}{\phi(q^2)}\ll
x\log\log x\sum_{y<q\le x^{1/2}}\frac{1}{q^2}\\
&& \ll \frac{x\log\log x}{y}=\frac{x}{(\log\log x)^2}=o(x)
\end{eqnarray*}
as $x\to\infty$, while the number of integers in the fourth and most
numerous category is
\begin{eqnarray*}
&& \le \sum_{y<q\le x^{1/2}}\sum_{\substack{p_1<p_2<x\\
p_i\equiv 1\pmod q,~i=1,2}}\left\lfloor
\frac{x}{p_1p_2}\right\rfloor\le x\sum_{y<q\le
x^{1/2}}\sum_{\substack{p_1<p_2<x\\ p_i\equiv 1\pmod
q,~i=1,2}}\frac{1}{p_1p_2}\\
&& \le x\sum_{y<q\le x^{1/2}}\frac{1}{2}\left(\sum_{\substack{p\le x\\
p\equiv 1\pmod q}}\frac{1}{p}\right)^2\ll x\sum_{y<q\le
x^{1/2}}\left(\frac{\log\log x}{\phi(q)}\right)^2\\
&& \ll x(\log\log x)^2 \sum_{y<q\le x^{1/2}}\frac{1}{q^2} \ll
\frac{x(\log\log x)^2}{y}=\frac{x}{\log\log x}=o(x)
\end{eqnarray*}
as $x\to\infty$. We now let
$$
{\cal Q}=\{p: t_p\le p^{1/3}\},
$$
and let ${\cal A}_2$ be the set of $k\le x$ divisible by some $q\in
{\cal Q}$ with $q>y$. To estimate $\#{\cal A}_2$, we begin by
estimating the counting function $\#{\cal Q}(t)$ of ${\cal Q}$ for
positive real numbers $t$.  Clearly,
$$
2^{\#{\cal Q}(t)}\le \prod_{q\in {\cal Q}(t)} q\le \prod_{s\le
t^{1/3}} (2^s-1)<2^{\sum_{s\le t^{1/3}} s}\le 2^{t^{2/3}},
$$
so
\begin{equation}
\label{eq:Q1} \#{\cal Q}(t)\le t^{2/3}.
\end{equation}
By Abel's summation formula, we now get that
$$
\#{\cal A}_2\le \sum_{\substack{y\le q\le x\\ q\in {\cal
Q}}}\left\lfloor \frac{x}{q}\right\rfloor\le x\sum_{\substack{y\le
q\le x\\ q\in {\cal Q}}}\frac{1}{q}\ll x\int_y^x\frac{d\#{\cal
Q}(t)}{t}\ll \frac{x}{y^{1/3}}=\frac{x}{\log\log x}=o(x)
$$
as $x\to\infty$.

Recall now that $P(m)$ stands for the largest prime factor of the
positive integer $m$. Known results from the theory of distribution
of smooth numbers show that uniformly for $3\le s\le t$, we have
\begin{eqnarray}
\label{eq:Psi} \Psi(t,s)=\#\{m\le t:P(m)\le s\}\ll t\exp(-u/2),
\end{eqnarray}
where $u=\log t/\log s$ (see \cite[Section III.4]{Ten}). Thus,
putting
$$
z=\exp\left(32(\log\log\log x)^2\right),
$$
we conclude that the estimate
\begin{equation}
\label{eq:psi} \Psi(t,y)\ll \frac{t}{(\log\log x)^{5}}
\end{equation}
holds uniformly for large $x$ once $t>z$, because in this case
$
u=\frac{\log t}{\log y}\ge \frac{32}{3}\log\log\log x,
$
therefore
$$
\frac{u}{2}\ge \frac{16}{3}\log\log\log x,
$$
so, in particular, $u/2>5\log\log\log x$ holds for all large $x$.
Furthermore, if $t>Z=\exp((\log\log x)^2)$, then
$$
u=\frac{\log t}{\log y}=\frac{(\log\log x)^2}{3\log\log\log x},
$$
so $u/2>2\log\log x$ one $x$ is sufficiently large. Thus, in this
range, inequality \eqref{eq:psi} can be improved to
\begin{equation}
\label{eq:psi1} \Psi(t,y)\ll \frac{x}{\exp(2\log\log x)}\ll
\frac{x}{(\log x)^2}.
\end{equation}
Now for a positive integer $m$, we put $d(m,y)$ for the largest
divisor $d$ of $m$ which is $y$-smooth, that is, $P(d)\le y$.
Let ${\cal A}_3$ be the set of $k\le x$ having a prime factor $p$
exceeding $z^{10}$ such that $d(p-1,y)>p^{1/10}$. To estimate
$\#{\cal A}_3$, we fix a $y$-smooth number $d$ and a prime $p$ with
$z^{10}<p<d^{10}$ such that $p\equiv 1\pmod d$, and observe that the
number of $n\le x$ which are multiples of this prime $p$ is $\le
\lfloor x/p\rfloor$. Note also that $d>p^{1/10}>z$. Summing up over
all the possibilities for $d$ and $p$, we get that $\#{\cal A}_3$
does not exceed
\begin{eqnarray*}
&& \sum_{\substack{z<d\\ P(d)\le y}}\sum_{\substack{p\le x\\ p\equiv
1\pmod d}}\left\lfloor \frac{x}{p}\right\rfloor\le
x\sum_{\substack{z<d\\ P(d)\le y}}\sum_{\substack{p\le x\\ p\equiv
1\pmod d}}\frac{1}{p} \ll x\sum_{\substack{z<d\\ P(d)\le
y}}\frac{\log\log x}{\phi(d)}\\
&& \ll x(\log\log x)^2\sum_{\substack{z<d\\ P(d)\le
y}}\frac{1}{d}\ll x(\log\log x)^2\int_{z}^x\frac{d\Psi(t,y)}{t}\\
&& \ll x(\log\log
x)^2\left(\frac{\Psi(t,y)}{t}\Big|_{z}^{x}+\int_z^x
\frac{\Psi(t,y)}{t^2} dt\right)\\
&& \ll \frac{x}{(\log\log x)^3}+ x(\log\log
x)^2\int_z^x\frac{\Psi(t,y)dt}{t^2}.
\end{eqnarray*}
In the above estimates, we used aside from the Abel summation
formula and inequality \eqref{eq:psi}, also the minimal order of the
Euler function $\phi(d)/d\gg 1/\log\log x$ valid for all $d\in
[1,x]$. It remains to bound the above integral. For this, we split
it at $Z$ and use estimates \eqref{eq:psi} and \eqref{eq:psi1}. In
the smaller range, we have that
$$
\int_z^{Z}\frac{\Psi(t,y)dt}{t^2}\ll \frac{1}{(\log\log
x)^5}\int_z^Z\frac{dt}{t}\ll \frac{\log Z}{(\log\log x)^5}\ll
\frac{1}{(\log\log x)^3}.
$$
In the larger range, we use estimate \eqref{eq:psi1} and get
$$
\int_Z^x\frac{\Psi(t,y)dt}{t^2}\ll \frac{1}{(\log x)^2}\int_Z^x
\frac{dt}{t}\ll \frac{1}{\log x}.
$$
Putting these together we get that
$$
\#{\cal A}_3\ll \frac{x}{(\log\log x)^3}+x(\log\log
x)^2\left(\frac{1}{(\log\log x)^3}+\frac{1}{\log x}\right)=o(x)
$$
as $x\to\infty$.

\medskip

Now let $\ell=d(k,z^{10})$. Put
$$
w=\exp(1920(\log\log\log x)^3),
$$
and put ${\cal A}_4$ for the set of $k\le x$ such that $\ell>w$.
Note that each such $k$ has a divisor $d>w$ such that $P(d)\le
z^{10}$. Since for such $d$ we have
$$
\frac{\log d}{\log (z^{10})}=6\log\log\log x,
$$
we get that in the range $t\ge w$,
$
u/2>3\log\log\log x,
$
for large $x$, so
\begin{equation}
\label{eq:psinew} \Psi(t,z^{10})<\frac{t}{(\log\log x)^3}
\end{equation}
uniformly for such $t$ once $x$ is large. Furthermore, if
$$
t>Z_1=\exp(1280\log\log x(\log\log\log x)^2),
$$
then
$
u=\frac{\log t}{\log z^{10}}>4\log\log x
$
therefore $u/2>2\log\log x$. In particular,
\begin{equation}
\label{eq:psi1new} \Psi(t,z^{10})\ll \frac{x}{(\log x)^2}
\end{equation}
in this range. By an argument already used previously, we have that
$\#{\cal A}_4$ is at most
\begin{eqnarray*}
&& \le \sum_{\substack{w<d<x\\ P(d)\le z^{10}}}\left\lfloor
\frac{x}{d}\right\rfloor\le x\sum_{\substack{w<d<x\\ P(d)\le
z^{10}}}\frac{1}{d}  \ll x\int_w^x \frac{d\Psi(t,z^{10})}{t}\\
&& \ll x
\left(\frac{\Psi(t,z^{10})}{t}\Big|_{t=w}^{t=x}+\int_w^x\frac{\Psi(t,z^{10})dt}{t^2}\right)
\\
&&\ll x\left(\frac{1}{(\log\log
x)^3}+\int_w^{Z_1}\frac{\Psi(t,z^{10})dt}{t^2}+\int_{Z_1}^x
\frac{\Psi(t,z^{10})dt}{t^2}\right)\\
&& \ll x\left(\frac{1}{(\log\log x)^3}+\frac{\log Z_1}{(\log\log
x)^3}+\frac{\log x}{(\log x)^2}\right)=o(x)
\end{eqnarray*}
as $x\to\infty$, where the above integral was estimated by splitting
it at $Z_1$ and using estimates \eqref{eq:psinew} and
\eqref{eq:psi1new} for the lower and upper ranges respectively.

\medskip

Let ${\cal A}_5$ be the set of $k\le x$ which are coprime to all
primes $p\in [y,z^{10}]$. By the Brun method,
$$
\#{\cal A}_5\ll x\prod_{y\le q\le z}\left(1-\frac{1}{q}\right)\ll
\frac{x\log y}{\log z}\ll \frac{x}{\log\log\log x} =o(x)
$$
as $x\to\infty$.

\medskip

We next let ${\cal A}_6$ be the set of $k\le x$ such that
$P(k)<w^{100}$. Clearly,
$$
\#{\cal A}_6=\Psi(x,w^{100})=x\exp\left(-c_1\frac{\log
x}{(\log\log\log x)^3}\right)=o(x)
$$
as $x\to\infty$, where $c_1=1/384000$.

\medskip

Finally, we let
$$
{\cal A}_7=\{k\le x: dp\mid k~{\text{\rm for~some}}~p\equiv 1\pmod
d~{\text{\rm and}}~p<d^3\}.
$$
Assume that $k\in {\cal A}_7$. Then there is a prime factor $p$ of
$k$ and a divisor $d$ of $p-1$ of size $d>p^{1/3}$ such that $dp\mid
k$. Fixing $d$ and $p$, the number of such $n\le x$ is $\le \lfloor
x/(dp)\rfloor$. Thus,
\begin{eqnarray*}
\#{\cal A}_7 & \le & \sum_{y\le p\le x}\sum_{\substack{d\mid p-1\\
d>p^{1/3}}}\left\lfloor \frac{x}{dp}\right\rfloor\le x\sum_{y\le
p\le
x}\frac{1}{p}\sum_{\substack{d\mid p-1\\ d>p^{1/3}}}\frac{1}{d}\\
& \ll & \sum_{y\le p\le x}\frac{1}{p}
\left(\frac{\tau(p-1)}{p^{1/3}}\right)\ll x\sum_{y\le p\le
x}\frac{\tau(p-1)}{p^{1+1/3}}\ll x\sum_{y\le p\le
x}\frac{1}{p^{5/4}}\\
& \ll & x\int_y^x \frac{dt}{t^{5/4}}\ll
\frac{x}{y^{1/4}}=\frac{x}{(\log\log x)^{3/4}}=o(x)
\end{eqnarray*}
as $x\to\infty$. Here, we used $\tau(m)$ for the number of divisors
of the positive integer $m$ and the fact that
$\tau(m)\ll_{\varepsilon} m^{\varepsilon}$ holds for all
$\varepsilon>0$ (with the choice of $\varepsilon=1/12$).

\medskip

{From} now on, $k\le x$ is odd and not in $\bigcup_{1\le i\le
7}{\cal A}_i$. From what we have seen above, most odd integers below
$x$ have this property. Then $\ell\le w$ because $k\not\in {\cal
A}_4$. Further, $k/\ell$ is square-free because $k\not\in {\cal
A}_1$. Moreover, if $p\mid k/\ell$, then $p>z^{10}>y$, therefore
$t_p>p^{1/3}$ because $k\not\in {\cal A}_2$. Since $k\not\in {\cal
A}_3$, we get that $d(p-1,y)<p^{1/10}$, so
$t_p'=t_p/\gcd(t_p,d(p-1,y))>p^{1/3-1/10}>p^{1/5}$ for all such $p$.
Moreover, $t_p'$ is divisible only by primes $>z>y$, so if $p_1$ and
$p_2$ are distinct primes dividing $k/\ell$, then $t_{p_1}'$ and
$t_{p_2}'$ are coprime because $k\not\in {\cal A}_1$. Finally,
$\ell>y$ because $k\not\in {\cal A}_5$. Furthermore, for large $x$
we have that $w>y$, so $k>\ell$ and in fact $k/\ell$ is divisible by
a prime $>w^{100}$ because $k\not\in {\cal A}_6$.

\medskip

We next put $n={\text{\rm lcm}}[d(\phi(k),y),\phi(\ell)]$. We let
$n_0$ stand for the minimal positive integer such that $n_0\equiv
-k+1\pmod {\phi(\ell)}$ and let $m=n_0+\ell\phi(\ell)$. Note that
$$
m\le 2\ell\phi(\ell)\le 2w^2=2\exp(3840(\log\log\log x)^3).
$$
We may also assume that $k>x/\log x$ since there are only at most
$x/\log x=o(x)$ positive integers $k$ failing this property. Since
$k>x/\log x$, we get that
$$
m<2\exp(3840(\log\log\log x)^3)<\lfloor \exp(4000(\log\log\log
k)^3)\rfloor=m(k)
$$
holds for large $x$. We will now show that this value for $m$ works.
First of all $m+k=n_0+\ell\phi(\ell)+k\equiv 1\pmod {\phi(\ell)}$ so
\begin{equation*}
\begin{split}
2^{m+k}-1 & \equiv 1\equiv
2^{\phi(\ell)-1}+2^{\phi(\ell)-2}+\cdots+2^{\phi(\ell)-(n_0-1)}+2^{\phi(\ell)-(n_0-1)}
+\\
& + 2^{x_1n}+\cdots+2^{x_tn}\pmod \ell,
\end{split}
\end{equation*}
where $t=\ell\phi(\ell)$ and $x_1,\ldots,x_t$ are any nonnegative
integers. Let
$$
U=2^{m+k}-1-2^{\phi(\ell)-1}-\cdots
-2^{\phi(\ell)-(n_0-1)}-2^{\phi(\ell)-(n_0-1)}.
$$
Then
$$
U\equiv \sum_{i=1}^t 2^{x_i\phi(\ell)}\pmod \ell
$$
for any choice of the integers $x_1,\ldots,x_t$. Let $p$ be any
prime divisor of $k/\ell$. Clearly,
$
\gcd(t_p,n)=d(t_p,y),
$
because $t_p\mid \phi(k)$ and $n\not\in {\cal A}_1$. In particular,
$$
t_p'=\frac{t_p}{\gcd(t_p,n)}\ge
\frac{t_p}{\gcd(d(\phi(k),y),p-1)}\ge p^{1/3-1/10}>p^{1/5}.
$$
Let $X=\{2^{jn}\pmod p\}$. Certainly, the order of $2^{n}$ modulo
$p$ is precisely $t_p'$. So, $\#X=t_p'>p^{1/5}$. A recent result of
Bourgain, Glibichuk and Konyagin (see Theorem 5 in \cite{BGK}),
shows that there exists a constant $T$ which is absolute such that
for all integers $\lambda$, the equation
$$
\lambda\equiv 2^{x_1n}+\cdots +2^{x_tn} \pmod p
$$
has an integer solutions $0\le x_1,\ldots,x_t<t_p'$ once $t>T$. In
fact, for large $p$ the number of such solutions
$$
N(t,p,\lambda)=\#\{(x_1,\ldots,x_t):0\le x_1,\ldots,x_t\le t_p\}$$
satisfies
$$
N(t,p,\lambda)\in \left[\frac{\#X^{t}}{2p},\frac{2\#X^t}{p}\right]
$$
independently in the parameter $\lambda$ and uniformly in the number
$t$. In particular, if we let $N_1(t,p,\lambda)$ be the number of
such solutions with $x_i=x_j$ for some $i\ne j$, then
$N_1(t,p,\lambda)\ll t^2\#X^{t-1}/p$. Indeed, the pair $(i,j)$ with
$i\ne j$ can be chosen in $O(t^2)$ ways, and the common value of
$x_i=x_j$ can be chosen in $\#X$ ways. Once these two data are
chosen, then the number of ways of choosing $x_s\in
\{0,1,\ldots,t_p'-1\}$ with $s\in \{1,2,\ldots,t\}\backslash\{
i,j\}$ such that
$$
\lambda-2^{x_in}-2^{x_jn}\equiv \sum_{\substack{1\le s\le t\\ s\ne
i,j}} 2^{x_sn}\pmod p
$$
is $N(t-2,p,\lambda-2^{x_in}-2^{x_jn})\ll \#X^{t-2}/p$ for $t>T+2$.
In conclusion, if all solutions $x_1,\ldots,x_t$ have two components
equal, then $p^{1/5}\ll \#X\ll t^2$, so $p\ll t^{10}$. For us, $t\le
2w^2$, so $p\ll  w^{20}$. Since $P(k)=P(k/\ell)>w^{100}$, it follows
that at least for the largest prime $p=P(k)$, we may assume that
$x_1,\ldots,x_t$ are all distinct modulo $p$ for a suitable value of
$\lambda$.

\medskip

We apply the above result with $\lambda=U$, $t=\ell\phi(\ell)$ (note
that since $t>y$, it follows that $t>T+2$ does indeed hold for large
values of $x$), and write ${\bf x}(p)=(x_1(p),\ldots,x_t(p))$ for a
solution of
$$
U\equiv 2^{x_1(p)n}+\cdots+2^{x_t(p)n}\pmod p,\qquad 0\le x_1(p)\le
\ldots\le x_t(p)<t_p'.
$$
We also assume that for at least one prime (namely the largest one)
the $x_i(p)$'s are distinct. Now choose integers $x_1,\ldots,x_t$
such that
$$
x_i\equiv x_i(p)\pmod {t_p'}
$$
for all $p\mid k/\ell$. This is possible by the Chinese Remainder
Lemma since the numbers $t_p'$ are coprime as $p$ varies over the
distinct prime factors of $k/\ell$. We assume that for each $i$,
$x_i$ is the minimal nonnegative integer in the corresponding
arithmetic progression modulo $\prod_{p\mid k/\ell} t_p'$. Further,
since $nx_i(p)$ are distinct modulo $t_p'$ when $p=P(k)$, it follows
that $nx_i$ are also distinct for $i=1,\ldots,t$. Hence, for such
$x_i$'s we have that
$
U-\sum_{i=1}^t 2^{x_in}
$
is a multiple of all $p\mid k/\ell$, and since $k/\ell$ is
square-free, we get that
$
U\equiv \sum_{i=1}^t 2^{x_in}\pmod {k/\ell}.
$
But the above congruence is also valid modulo $\ell$, so it is valid
modulo $k={\text{\rm lcm}}[\ell,k/\ell]$, since $\ell$ and $k/\ell$
are coprime. Thus,
$$
U\equiv \sum_{i=1}^t 2^{x_in}\pmod k,
$$
or
$$
2^{k+m-1}-1\equiv
2^{\phi(\ell)-1}+\cdots+2^{\phi(\ell)-(n_0-1)}+2^{\phi(\ell)-(n_0-1)}+\sum_{i=1}^t
2^{x_i n}\pmod k.
$$
As we have said, the numbers $x_in$ are distinct and they can be
chosen of sizes at most $n{\text{\rm lcm}}[t_p':p\mid k\ell]\le
\phi(k)\le k$. Finally, $nx_i$ are divisible by $\phi(\ell)$ whereas
none of the numbers $\phi(\ell)-j$ for $j=1,\ldots,n_0-1$ is unless
$n_0=1$. Thus, assuming that $n_0\ne 1$, we get that all the
$m=t+n_0$ exponents are distinct except for the fact that
$\phi(\ell)-(n_0-1)$ appears twice. Let us first justify that
$n_0\ne 1$. Recalling the definition of $n_0$, we get that if this
were so then $\phi(\ell)\mid k$. However, we have just said that
$\ell$ has a prime factor $p>y$. If $\phi(\ell)\mid k$, then $k$ is
divisible by both $p$ and $p-1$ for some $p>y$ and this is
impossible since $n\not\in {\cal A}_1$. Finally, to deal with the
repetition of the exponent $\phi(\ell)-(n_0-1)$, we replace this by
$\phi(\ell)-(n_0-1)+t_k$, where as usual $t_k$ is the order of $2$
modulo $n$. We show that with this replacement, all the exponents
are distinct. Indeed, this replacement will not change the value of
$2^{\phi(\ell)-(n_0-1)+t_k}\pmod k$. Assume that after this
replacement, $\phi(\ell)-(n_0-1)+t_k$ is still one of the remaining
exponents. If it has become a multiple of $n$, it follows that it is
in particular divisible by $t_p$ for all primes $p\mid \ell$. Since
$t_p\mid t_k$ and $t_p\mid \phi(\ell)$ for all primes $p\mid \ell$,
we get that $t_p\mid n_0-1$, so $t_p\mid k$. Since $\ell$ is
divisible by some prime $p>y$ (because $k\not\in {\cal A}_5$), we
get that $t_p\mid k$. Since $k\not\in {\cal A}_2$, we get that
$t_p>p^{1/3}$. Thus, $k$ is divisible by a prime $p>y$ and a divisor
$d$ of $p-1$ with $d>p^{1/3}$, and this is false since $n\not\in
{\cal A}_7$. Hence, this is impossible, so it must be the case that
$\phi(\ell)-(n_0-1)+t_k\in
\{\phi(\ell)-1,\ldots,\phi(\ell)-(n_0-1)\}$. This shows that $t_k\le
n_0\le \ell\phi(\ell)\le 2^{10}w^{20}$. However, $t_k$ is a multiple
of $t_{P(k)}\ge P(k)^{1/3}$, showing that $P(k)\le 2^{30} w^{60}$,
which is false for large $x$ since $k\not\in {\cal A}_6$. Thus, the
new exponents are all distinct for our values of $k$. As far as
their sizes go, note that since $k$ has at least two odd prime
factors, it follows that $t_k\mid \phi(k)/2$, therefore
$\phi(\ell)-(n_0-1)+t_k\le w+\phi(k)/2<w+k/2<k$ since $k>2w$ for
large $x$. Thus, we have obtained a representation of $2^{k+m}-1$
modulo $m$ of the form
$$
2^{j_1}+\cdots+2^{j_m}\pmod k
$$
where $0\le j_1<\ldots<j_m\le k$, which shows that $k\in {\cal
C}_m$. Since $m\le m(k)$ and ${\cal C}_m\subset {\cal C}_{m(k)}$,
the conclusion follows.
\end{proof}

\begin{rem}
\label{rem9} The above proof shows that in fact the number of
odd $k<x$ such that $k\not \in {\cal C}_{m(k)}$ is $O(x/\log\log\log
x)$.
\end{rem}

\begin{corollary}
\label{cor:X} For large $x$, the inequality $c_k/2^k<2^{m(k)-(\log
k)/(\log 2)}$ holds for all odd $k<x$ with at most $O(x/\log\log\log
x)$ exceptions.
\end{corollary}

\begin{proof} This follows from the fact that $c_k=a_k/k\le
2^{k+m}/k$, where $k\in {\cal C}_m$ (see Theorem \ref{cm}), together
with above Theorem \ref{thm:2} and Remark~\ref{rem9}.
\end{proof}

\begin{corollary}
\label{cor:1} The estimate
$$
\frac{1}{x}\sum_{\substack{1\le k\le x\\
k~{\text{\rm odd}}}}\frac{c_k}{2^k}=O\left(\frac{1}{\log\log\log
x}\right)
$$
holds for all $x$.
\end{corollary}

\begin{proof}
If $k\le x/\log x$ is odd, then $c_k/2^k\le 1$, so
$$
\sum_{\substack{k\le x/\log x\\ k~{\text{\rm
odd}}}}\frac{c_k}{2^k}\le \frac{x}{\log x}.
$$
If $k\in [x/\log x,x]$ but $k\not\in {\cal C}_{m(k)}$, then still
$c_k/2^k\le 1$ and, by the Corollary \ref{cor:X}, the number of such
$k$'s is $O(x/\log\log\log x)$. Thus,
$$
\sum_{\substack{k\in [x/\log x,x]\\ k\not\in {\cal C}_{m(k)}\\
k~{\text{\rm odd}}}}\frac{c_k}{2^k}\ll \frac{x}{\log\log\log x}.
$$
For the remaining odd values of $k\le x$, we have that
$$
\frac{c_k}{2^k}\le 2^{m(k)-(\log k)/(\log 2)},
$$
so it suffices to show that
$$
2^{m(k)-(\log k)/(\log 2)}<\frac{1}{\log\log\log x},
$$
is equivalent to
$$
(\log k)/(\log 2)-m(k)>\log\log\log\log x/\log 2,
$$
which in turn is implied by
$$
\log(x/\log x)-\log\log\log\log x>(\log 2)\exp(4000(\log\log\log
x)^3),
$$
and this is certainly true for large $x$. Thus, indeed,
$$
\sum_{\substack{1\le k\le x\\
k~{\text{\rm odd}}}}\frac{c_k}{2^k}=O\left(\frac{x}{\log\log\log
x}\right),
$$
which is what we wanted to prove.
\end{proof}

In particular,

\begin{equation}\label{averageconjecture}
\frac{1}{x}\sum_{\substack{k\le x\\
k~{\text{\rm odd}}}}\frac{c_k}{2^k}=o(1)
\end{equation}
as $x\to\infty$. One can adapt these techniques to obtain that the
whole sequence $c_k/2^k$ is convergent to $0$ in arithmetic average.
In order to do so, the sets ${\cal C}_m$ should be suitably modified
and an analog of Theorem \ref{thm:2} for these new sets should be
proved. We leave this for a subsequent work.

\section{Existence and bounds for $a_k$ in base $q> 2$}
\label{sec:7}

Let $q\ge 2$ be a fixed integer and let $x$ be a positive real
number. Put
 \begin{eqnarray*}
 V_k(x)&=&\{0\le n<x: s_q(n)=k\},\\
 V_k(x;h,m)&=&\{0\le n<x: s_q(n)=k, n\equiv h\pmod m\}.
 \end{eqnarray*}
Mauduit and S\'ark\"ozy proved in \cite{MS} that if
 $\gcd(m,q(q-1))=1$, then there exists some constant $c_0$ depending on $q$
 such that if we put
$$
\ell=\min\left\{k,(q-1)\lfloor \log x/\log q\rfloor-k\right\},
$$
then $V_k(x)$ is well distributed in residues classes modulo $m$
provided that $m<\exp(c_0 \ell^{1/2})$.

Taking $m=k$ and $h=0$, we deduce that if $k<\exp(c_0 \ell^{1/2})$,
then
$$
V_k(x;0,k)=(1+o(1))V_k(x)/k
$$
as $x\to\infty$ uniformly in our range for $k$. The condition on $k$
is equivalent to $\log k\ll \ell^{1/2}$, which is implied by
$k+O((\log k)^2)\ll \log x$. Thus, we have the following result.

\begin{lemma}
\label{lem1} Let $q\ge 2$ be fixed. There exists a constant $c_1$
such that if $k$ is any positive integer with $\gcd(k,q(q-1))=1$,
then $V_k(x)$ is well distributed in arithmetic progressions of
modulus $k$ whenever $x>\exp(c_1 k)$.
\end{lemma}

Corollary 2 of \cite{MS} implies that if
\begin{equation}
\label{cond-delta} \Delta=\left|\frac{q-1}{2\log q}\log
x-k\right|=o(\log x)\qquad {\text{\rm as}}~x\to\infty,
\end{equation}
then the estimate
\[
\#V_k(x)= \frac{x}{(\log x)^{1/2}}\exp\left(-c_3\frac{\Delta^2}{\log
x}+ O\left(\frac{\Delta^3}{(\log x)^2}+\frac{1}{(\log
x)^{1/2}}\right) \right)
\]
holds with some explicit constant $c_3$ depending on $q$. As a
corollary of this result, we deduce the following result.
\begin{lemma}
\label{lem2} If condition \eqref{cond-delta} is satisfied, then
$V_k(x)\not= \emptyset$.
\end{lemma}

In case $k$ and $q$ are coprime but $k$ and $q-1$ are not, we may
apply instead Theorem B of \cite{MPS} with $m=k$ and $h=0$ to arrive
at a similar result.

\begin{lemma}
\label{lemB} Assume that $q\ge 2$ is fixed. There exists a constant
$c_4$ depending only on $q$ such that if $k$ is a positive integer
with $\gcd(k,q)=1$, and $x\ge \exp(c_4 k)$, then $V_k(x;0,k)\ne 0$.
\end{lemma}

One can even remove the coprimality condition on $q$ and $k$. Assume
that $x$ is sufficiently large such that
\begin{equation}
\label{eq:Delta1} \Delta\le c_5(\log x)^{5/8},
\end{equation}
where $c_5$ is some suitable constant depending on $q$. Using
Theorem~C and Lemma 5 of \cite{MPS} with $m=k$ and $h=0$, we obtain
the following result.

\begin{lemma}
\label{lemC} Assume that both estimates \eqref{eq:Delta1} and
$k<2^{(\log x)^{1/4}}$ hold. Then $V_k(x;0,k)\ne \emptyset$.
\end{lemma}

A sufficient condition on $x$ for Lemma \ref{lemC} above to hold is
that $x>\exp(c_6 k)$, where $c_6$ is a constant is a constant that
depends on $q$. Putting Lemmas~\ref{lemB} and \ref{lemC} together we
obtain the next theorem.
\begin{thm}
\label{thm_q} For all $q\ge 2$ there exists a constant $c_6$
depending on $q$ such that for all $k\ge 1$ there exists $n\le
\exp(c_6 k)$ with $s_q(kn)=k$.
\end{thm}

Consequently, $a_k=\exp(O(k))$ for all $k$, and in particular it is
nonzero. The following example of Lemma \ref{lem6} shows that
$a_k=\exp(o(k))$ does not always hold as $k\to\infty$.
 \begin{lem}
\label{lem6} If $q>2$, then
\begin{equation*}
a_{q^m}=q^m \left(2 q^{\frac{q^m-1}{q-1}}-1  \right).
\end{equation*}
If $q=2$, then  $a_{2^m}=2^m(2^{2^m}-1)$.
\end{lem}
\begin{proof}
The fact that $s_q(a_{q^m})=q^m$ for all $q\ge 2$ is immediate. We
now show the minimality of the given $a_{q^m}$ with this property.
Let $\alpha_m=(q^m-1)/(q-1)$. Note that every digit of
$q^{\alpha_m}-1$ in base $q$ is maximal, so $q^{\alpha_m}-1$ is
minimal such that $s_q(q^{\alpha_m}-1)=q^m-1$. Since
\[
q^{\alpha_m}-1=(q-1) q^{\alpha_m-1}+(q-1) q^{\alpha_m-2}+\cdots
+(q-1),
\]
then  $a_{q^m}$ must contain the least term $q^t$, where
$t>\alpha_m-1$ such that its sum of digits is $q^m$ and
$q^m|a_{q^m}$. The least term  is obviously $q^{\alpha_m}$, and it
just happens that  $a_{q^m}$ such defined satisfies the mentioned
conditions.
\end{proof}

\end{document}